\documentclass[11pt]{amsart}
\usepackage{amsrefs,amsmath,array,enumerate,tikz,bm,graphicx,comment}

\newcommand{\A}{\mathbb{A}}
\newcommand{\PP}{\mathbb{P}}
\newcommand{\Q}{\mathbb{Q}}
\newcommand{\Z}{\mathbb{Z}}

\newcommand{\Cn}{\langle C_n\rangle}
\newcommand{\Qbar}{\overline{\mathbb{Q}}}

\newcommand{\calB}{\mathcal{B}}
\newcommand{\calE}{\mathcal{E}}
\newcommand{\calM}{\mathcal{M}}
\newcommand{\calS}{\mathcal{S}}

\newcommand{\Aut}{\operatorname{Aut}}

\newcommand{\Per}{\operatorname{Per}}
\newcommand{\PGL}{\operatorname{PGL}}
\newcommand{\PrePer}{\operatorname{PrePer}}
\newcommand{\Rat}{\operatorname{Rat}}
\newcommand{\Res}{\operatorname{Res}}
\newcommand{\Spec}{\operatorname{Spec}}

\newtheorem{thm}{Theorem}[section]
\newtheorem{lem}[thm]{Lemma}
\newtheorem{prop}[thm]{Proposition}

\theoremstyle{definition}
\newtheorem{rem}[thm]{Remark}

\numberwithin{equation}{section}

\title[Portraits of quadratic rational maps]{Portraits of quadratic rational maps with a small critical cycle}

\author[T. Dunaisky]{Tyler Dunaisky}
\address{Department of Mathematics, Purdue University, Indiana, USA}
\email{tdunaisk@purdue.edu}

\author[D. Krumm]{David Krumm}
\address{School of Mathematics, University of Costa Rica, San Jos\'e, Costa Rica}
\email{david.krumm@ucr.ac.cr}
\urladdr{http://maths.dk}

\keywords{Periodic point, critical point, dynamical modular curve}
\subjclass{Primary 37P05; Secondary 11G30}

\begin{document}
\maketitle

\begin{abstract}
Motivated by a uniform boundedness conjecture of Morton and Silverman, we study the graphs of pre-periodic points for maps in three families of dynamical systems, namely the collections of rational functions of degree two having a periodic critical point of period $n$, where $n\in\{2,3,4\}$. In particular, we provide a conjecturally complete list of possible graphs of rational pre-periodic points in the case $n=4$, analogous to well-known work of Poonen for $n=1$, and we strengthen earlier results of Canci and Vishkautsan for $n\in\{2,3\}$. In addition, we address the problem of determining the representability of a given graph in our list by infinitely many distinct linear conjugacy classes of maps.
\end{abstract}

\section{Introduction}\label{intro_section}
Let $k$ be a field of characteristic $0$ with algebraic closure $\bar k$. For every nonconstant rational function $\phi\in k(x)$, the \textbf{degree} of $\phi$ is the integer $\deg(\phi)=\max\{\deg p,\deg q\}$, where $p,q$ are coprime polynomials such that $\phi=p/q$. We will write $\Rat_d(k)$ for the set of all rational functions of a fixed degree $d\ge 1$ having coefficients in $k$. Let $\PP^1:=\PP^1_k$ be the projective line over $k$. Corresponding to the rational map $\phi$ there is a morphism of algebraic varieties $\PP^1\to\PP^1$, which we also denote by $\phi$; we may thus regard rational functions as discrete dynamical systems on $\PP^1$. A \textbf{critical point} of $\phi$ is a point of $\PP^1$ where the corresponding morphism is ramified.

Writing $\phi^i$ for the $i$th iterate of $\phi$, a point $P\in\PP^1(\bar k)$ is called \textbf{periodic} for $\phi$ if $\phi^n(P)=P$ for some positive integer $n$; in that case, the least such $n$ is the \textbf{period} of $P$, and $P$ is called an $\bm{n}$\textbf{-periodic point} of $\phi$. More generally, $P$ is \textbf{pre-periodic} for $\phi$ if there exists $m\ge 0$ such that $\phi^m(P)$ is periodic. In that case, the \textbf{pre-periodic type} of $P$ is the pair $(m,n)$, where $n$ is the eventual period of $P$ and $m$ is the least positive integer such that $\phi^m(P)$ is $n$-periodic.

For any field extension $K/k$ with $K\subseteq\bar k$, let $\PrePer(\phi,K)$ denote the set of all $K$-rational points of $\PP^1$ that are pre-periodic for $\phi$. This set admits a natural structure of directed graph in which vertices correspond to pre-periodic points and edges have the form $P\to\phi(P)$ for $P\in\PrePer(\phi,K)$. The isomorphism class of the digraph $\PrePer(\phi,K)$ will be called the \textbf{portrait} of $\phi$ over $K$, and denoted $G(\phi,K)$.

Two rational functions $\phi, \psi\in k(x)$ are called \textbf{linearly conjugate} over the field $K$ if there exists $\sigma\in\Rat_1(K)$ such that $\psi=\sigma^{-1}\circ\phi\circ\sigma$; in that case $\phi$ and $\psi$ are considered to be equivalent as dynamical systems on $\PP^1_K$, and in particular they have the same portrait over $K$. The linear conjugacy class of $\phi$ over $\bar k$ will be denoted $[\phi]$ and called the \textbf{geometric dynamical system} of $\phi$. The conjugacy class of $\phi$ over $k$ will be denoted $\langle\phi\rangle$ and called the \textbf{arithmetic dynamical system} of $\phi$. Finally, the \textbf{automorphism group} of $\phi$, denoted $\Aut(\phi)$, consists of all functions $\sigma\in\Rat_1(\bar k)$ satisfying $\sigma^{-1}\circ\phi\circ\sigma=\phi$. 

\subsection{Families of quadratic dynamical systems} The present article is motivated by a wide-ranging conjecture of Morton and Silverman \cite{morton-silverman} in arithmetic dynamics, one consequence of which is that, for every degree $d\ge 2$, the set of portraits $\{G(\phi,\Q):\phi\in\Rat_d(\Q)\}$ is finite. This weaker statement has not been proved, even for $d=2$, but some progress has been made by restricting attention to one-parameter families of maps. Notably, the combined results of several articles \cites{poonen,morton_period4,flynn-poonen-schaefer,stoll,hutz-ingram} provide compelling evidence that there exist exactly 12 portraits of the form $G(x^2+c,\Q)$ with $c\in\Q$, namely those listed by Poonen in \cite{poonen}*{Figure 1}. Our main contribution in this paper is to provide similar lists of portraits for other one-parameter families of rational maps of degree two, namely the collections of maps $\phi\in\Rat_2(\Q)$ having an $n$-periodic critical point, where $n\in\{2,3,4\}$.

As explained in \textsection\ref{M2_section}, Milnor \cite{milnor} showed that the set of all degree-2 dynamical systems on $\PP^1$ is in bijection with the affine plane $\A^2$. Moreover, under this correspondence, dynamical systems having an $n$-periodic critical point correspond to points on an algebraic curve in $\A^2$, which Milnor denotes by $\Per_n(0)$ and we will denote by $C_n$ in order to ease notation. When convenient, we will identify points of $C_n$ with the corresponding dynamical systems on $\PP^1$. The curve $C_n$ is known to be rational for $n\le 4$, so its set of rational points is infinite and parametrizable; in contrast, Ramadas and Silversmith \cite{ramadas-silversmith}*{Corollary 1.4} show that $C_5$ has no rational point. Parametrizations of $C_n$ for $n\in\{1,2,3\}$ are provided by Milnor (see Lemma 3.4 and Example D.3 in \cite{milnor}); a parametrization of $C_4$ is derived here in \textsection \ref{arithmetic_parametrization} as an application of a method for computing equations for the curves $C_n$.

\subsection{Main results} This paper concerns the portraits of dynamical systems in $C_n$ for $n\in\{2,3,4\}$. Work in this direction was initiated by Poonen's article \cite{poonen}, which includes a complete list of portraits $G(\phi,\Q)$ for maps $\phi\in C_1(\Q)$, under the assumption that no such map $\phi$ has a rational periodic point of period greater than three. Analogous results for $n=2$ and $3$ were obtained by Canci and Vishkautsan \cites{canci-vishkautsan,vishkautsan}. Theorem \ref{main_thm} below is a natural continuation of this progression of results, handling the case $n=4$.

In stating our main theorem we use the terminology and results established by Silverman in \cite{silverman_FOD}; in particular, every nonconstant rational function $\phi\in\bar k(x)$ has a corresponding \textbf{field of moduli}, denoted by $k_{\phi}$, which has the property that $k_{\phi}=k_{\rho}$ if $\rho\in[\phi]$. Moreover, if $\phi$ has even degree, the field $k_{\phi}$ can be described as the smallest intermediate field $K$ in the extension $\bar k/k$ such that $\phi\in[\psi]$ for some rational function $\psi\in K(x)$.

\begin{thm}\label{main_thm}
Suppose that $\phi\in\Qbar(x)$ is a rational function of degree two having a $4$-periodic critical point in $\PP^1(\Qbar)$.
\begin{enumerate}[(a)]
    \item\label{main_thm_nontrivial} If $\Aut(\phi)$ is nontrivial, then $\Q_{\phi}$ is a number field of absolute degree two or six; in particular, $\phi$ cannot have rational coefficients.
    \item\label{main_thm_trivial} If $\phi\in\Q(x)$, then $\phi$ has a unique $4$-periodic critical point, and this point is defined over $\Q$. Moreover, assuming that $\phi$ has no rational periodic point of period greater than four, and assuming known all rational points on three algebraic curves --- two of genus five and one of genus six --- the portrait $G(\phi,\Q)$ must be one of the five shown in Figure \ref{period4portraits_trivial}.
\end{enumerate} 
\end{thm}

Empirical data gathered here (see Proposition \ref{portrait_search} and Remark \ref{period4_rem}) suggests that no map $\phi\in\Rat_2(\Q)$ which has a $4$-periodic critical point can have a rational periodic point of period greater than four. Assuming this is indeed the case, and assuming we have found all rational points on the three curves mentioned in Theorem \ref{main_thm}\eqref{main_thm_trivial}, the theorem implies that the five portraits in Figure \ref{period4portraits_trivial} represent all possible portraits $G(\phi,\Q)$ for such maps $\phi$. Thus, since the curve $C_4$ has infinitely many rational points, at least one of these portraits must correspond to infinitely many distinct arithmetic dynamical systems $\langle\phi\rangle$. Our next result identifies precisely which of the five portraits arise infinitely often.

\begin{thm}\label{infinitely_represented}
   The portraits labelled {\rm F1} and {\rm F2} in Figure \ref{period4portraits_trivial} occur as $G(\phi,\Q)$ for only finitely many distinct arithmetic dynamical systems $\langle\phi\rangle$ of maps $\phi\in\Rat_2(\Q)$ such that $[\phi]\in C_4$. Moreover, assuming that no such map $\phi$ has a rational periodic point of period greater than four, and assuming known all rational points on the curves referenced in Theorem \ref{main_thm}\eqref{main_thm_trivial}, the portraits {\rm I1}, {\rm I2}, and {\rm I3} arise as $G(\phi,\Q)$ for infinitely many distinct arithmetic dynamical systems $\langle\phi\rangle$ with $\phi$ as above.
\end{thm}

Theorem \ref{cv_extensions} below summarizes our results for maps of degree two having a periodic critical point of period two or three. This theorem serves as a supplement to the work of Canci and Vishkautsan \cites{canci-vishkautsan,vishkautsan}, and extends some of the results of these authors by taking into consideration all periodic critical points regardless of their fields of definition.

\begin{thm}\label{cv_extensions} Let $n\in\{2,3\}$, and suppose that $\phi\in\Qbar(x)$ is a rational function of degree two having an $n$-periodic critical point in $\PP^1(\Qbar)$.
\begin{enumerate}[(a)]
    \item\label{3_nontrivial} If $n=3$ and $\Aut(\phi)$ is nontrivial, then $\Q_{\phi}$ is a number field of absolute degree four; in particular, $\phi$ cannot have rational coefficients.
    \item\label{3_trivial} If $n=3$ and $\phi\in\Q(x)$, then $\phi$ has a $3$-periodic critical point defined over $\Q$. Hence, under the assumptions of Theorem 1 in \cite{vishkautsan}, the portrait $G(\phi,\Q)$ must be one of the six referred to in the theorem.
    \item\label{2_trivial} Suppose $n=2$, $\phi\in\Q(x)$, and $\Aut(\phi)$ is trivial. Then $\phi$ has a unique $2$-periodic critical point, and this point is defined over $\Q$. In particular, under the assumptions of Theorem 1.2 in \cite{canci-vishkautsan}, the portrait $G(\phi,\Q)$ must be one of the $13$ referred to in the theorem.
    \item\label{2_nontrivial} Suppose $n=2$, $\phi\in\Q(x)$, and $\Aut(\phi)$ is nontrivial. Then $\phi$ has exactly two $2$-periodic critical points, and these points generate a number field $K$ of absolute degree at most two. Moreover, the portrait $G(\phi,K)$ must be one of the four shown in Figure \ref{period2portraits}.
\end{enumerate}
\end{thm}

\subsection{Methods} The main construction used in this article is the definition of a collection of dynamical modular curves corresponding to any nonconstant morphism $\phi:\PP^1_k\to\PP^1_k$, where $k=\Q(t)$ is a function field in one variable. A process of reduction applied to $\phi$ yields a one-parameter family of specialized maps $\phi_c:\PP^1_{\Q}\to\PP^1_{\Q}$ with $c\in\Q$, and various arithmetic dynamical properties of these maps can be determined by analyzing the sets of rational points on the modular curves we define. Parts of our construction of dynamical modular curves may be regarded as generalizations of earlier work by Morton \cite{morton_curves}, which applies to polynomial mappings $\phi$.

For the purposes of computing rational points on curves, the main tools used here are the \texttt{Chabauty} function in \textsc{Magma} \cite{magma} and the \texttt{Chabauty} package of  Balakrishnan and Tuitman \cite{balakrishnan-tuitman}. For computing the portraits of rational functions over $\Q$, as well as other dynamical invariants, we use several tools included in \textsc{Sage} \cite{sagemath}, most of them due to Hutz \cites{hutz,hutz_dynatomic,hutz_multipliers}.

For studying the arithmetic dynamics of quadratic maps with a periodic critical point, a key method developed in the paper is an algorithm for computing an equation for any curve $C_n$. The method does not seem to appear elsewhere in the literature, though the details will be clear to the expert. Combined with the machinery of dynamical modular curves, this algorithm allows us to classify the rational portraits of maps in $C_4$.

With a view towards future applications of our methods, we note that a similar process to that carried out for $C_4$ could, in principle, be used to study the rational portraits in other one-parameter families of rational maps. 

\subsection{Outline} This article is organized as follows. In Section \ref{arithmetic_parametrization} we describe a method for computing equations for the curves $C_n$, and we apply this method to derive parametrizations of the sets of arithmetic dynamical systems $\langle\phi\rangle$, where $\phi\in\Rat_2(\Q)$ has a periodic critical point of period $n\in\{2,3,4\}$. Section \ref{period2_3_section} is devoted to the proof of Theorem \ref{cv_extensions}. In Section \ref{dmc_section} we define several types of dynamical modular curves associated to a one-parameter family of maps, and in Section \ref{period4_section} we use the modular curves for the family $C_4$ in order to prove Theorems \ref{main_thm} and \ref{infinitely_represented}. Appendix \ref{realization_appendix} includes a proof of a lemma used to realize every portrait in Figure \ref{period2portraits} in the form $G(\phi,K)$ as in Theorem \ref{cv_extensions}\eqref{2_nontrivial}.

\subsection{Computer code} All computations needed for this article were carried out using \textsc{Magma} V2.26-12 or \textsc{Sage} V10.2. The code used is available in \cite{code}, with files named according to the relevant subsections of the article. 

\section{Parametrization of arithmetic dynamical systems}\label{arithmetic_parametrization}
In this section we classify the maps $\phi\in\Rat_2(\Q)$ having an $n$-periodic critical point, where $n\in\{2,3,4\}$, up to linear conjugacy over $\Q$. Our main results in this direction are Propositions \ref{parametrization2}-\ref{parametrization4}.

\subsection{The moduli space of quadratic dynamical systems}\label{M2_section} Let $k$ be a field of characteristic $0$ with algebraic closure $\bar k$. For every positive integer $d$, set $\Rat_d:=\Rat_d(\bar k)$, $\PGL_2:=\PGL_2(\bar k)$. The group $\PGL_2$, identified with $\Rat_1$, acts on $\Rat_d$ by conjugation, so that the $\PGL_2$-orbit of a map $\phi\in\Rat_d$ is the geometric dynamical system $[\phi]$. As shown by Silverman \cite{silverman_Md}  (see also \cite{silverman_book}*{\textsection 4.4}), there exists an algebraic variety $\calM_d$, defined over $\Q$, whose points correspond to dynamical systems of degree $d$; more precisely, there is a bijection $\calM_d(\bar k)\stackrel{\sim}{\longrightarrow}\Rat_d/\PGL_2$. As will be illustrated here in the case $d=2$, this correspondence can be used to translate some dynamical properties of rational maps into equations defining subvarieties of $\calM_d$. 

Recall that if $P\in\PP^1(\bar k)$ is an $n$-periodic point of $\phi\in\Rat_d$, the \textbf{multiplier} of $\phi$ at $P$ is given by $\lambda_{\phi,P}:=(\phi^n)'(P)$, possibly after conjugating $\phi$ in order to move $P$ away from $\infty$. The $\bm{n}$\textbf{-multiplier spectrum} of $\phi$ is the multiset
\[\Lambda_n(\phi)=\{\lambda_{\phi,P}:P\in\Per_n(\phi)\},\]
where $\Per_n(\phi)$ denotes the set of $n$-periodic points of $\phi$ in $\PP^1(\bar k)$. A straightforward argument shows that $\Lambda_n(\phi)$ is a well-defined invariant of the dynamical system $[\phi]\in\calM_d$; furthermore, multiplier spectra have the following property concerning critical points:
\begin{equation}\label{Lambda_property}
    \phi \text{ has an $n$-periodic critical point} \iff 0\in\Lambda_n(\phi).
\end{equation}

 By using multipliers at fixed points, Milnor \cite{milnor} (see also \cite{silverman_book}*{Theorem 4.56}) showed that the variety $\calM_2$ is isomorphic to $\A^2$ over $\Q$. Explicitly, there is an isomorphism
\begin{equation}\label{M2A2_iso}
\calM_2\stackrel{\sim}{\longrightarrow}\A^2,\quad[\phi]\mapsto(\sigma_1,\sigma_2),
\end{equation}
 where $\sigma_i$ is the $i$th elementary symmetric function of the elements of $\Lambda_1(\phi)$. The pair $(\sigma_1,\sigma_2)$ will be called the \textbf{coordinates} of $[\phi]$. As noted in \cite{silverman_book}*{Remark 4.83}, the field of moduli $k_{\phi}$ is equal to the field $k(\sigma_1,\sigma_2)$ generated by the coordinates of $\phi$. In practice, these coordinates can be computed using more general algorithms due to Hutz \cite{hutz_multipliers} which are included in \textsc{Sage}.

The \textbf{symmetry locus} in $\calM_2$ is the set
 \[\calS=\{[\phi]\in\calM_2:\Aut(\phi)\text{ is nontrivial}\}.\]
Milnor shows that, under the isomorphism \eqref{M2A2_iso}, the set $\calS$ corresponds to a cuspidal cubic curve, which we also denote by $\calS$. Applying the parametrization given by Milnor \cite{milnor}*{Corollary 5.3}, and using $(r,s)$ as coordinates on $\A^2$, the following equation for $\calS$ is obtained in \cite{manes-yasufuku}*{p.\ 261}:
\begin{equation}\label{symlocuseq}
-2r^3-r^2s+r^2+8rs+4s^2-12r-12s+36=0.
\end{equation}
 
For $n\ge 1$, let $C_n$ denote the subset of $\calM_2(\bar k)$ consisting of dynamical systems $[\phi]$ such that $\phi$ has an $n$-periodic critical point in $\PP^1(\bar k)$. Work of Milnor \cite{milnor}*{Corollary D.2} and Silverman \cite{silverman_Md}*{Corollary 5.2} shows that the image of $C_n$ under the isomorphism \eqref{M2A2_iso} is the set of points on an algebraic curve in $\A^2$. When convenient, we will identify $C_n$ with this curve.

\subsection{Equations for the curves $C_n$}\label{Cn_eqn_section}  We now restrict to working over the field $k=\Q$, so that the varieties $\calM_2=\A^2,\calS,$ and $C_n$ are all defined over $\Q$. In this subsection we show how to compute an equation for the curve $C_n$, our main result being Proposition \ref{Cn_equation}. Using this method we recover the parametrizations of $C_1$, $C_2$, and $C_3$ obtained by Milnor \cite{milnor}, and we derive a parametrization of $C_4$.   

Let $K$ be a field of characteristic $0$ and $\phi\in K(x)$ a nonconstant rational function. For every positive integer $n$, the $n$th \textbf{dynatomic polynomial} of $\phi$, denoted $\Phi_{n,\phi}$, is defined by
\[\Phi_{n,\phi}:=\prod_{d|n}(x\cdot q_d-p_d)^{\mu(n/d)},\]
where $\mu$ is the classical M{\"o}bius function and $p_d,q_d\in K[x]$ are coprime polynomials such that $\phi^d=p_d/q_d$. The rational function thus defined is in fact a polynomial (see \cite{silverman_book}*{\textsection 4.1}) and has the following property: identifying $\PP^1(\bar K)$ with $\bar K\cup\{\infty\}$, every $n$-periodic point of $\phi$ in $\bar K$ is a root of $\Phi_{n,\phi}$, and conversely, every \textit{simple} root of $\Phi_{n,\phi}$ is $n$-periodic for $\phi$ (see \cite{hutz_dynatomic}*{\textsection 3}).

\begin{lem}\label{Un_lem}
 With $K$ and $\phi$ as above, write $\phi'=p/q$ with $p,q\in K[x]$ coprime polynomials. Suppose that $\Phi_{n,\phi}$ has nonzero discriminant and that $\infty$ is not $n$-periodic for $\phi$. Let $\ell\in K$ be the leading coefficient of $\Phi_{n,\phi}$, and define $U_n(\phi)\in K$ by
\[ U_n(\phi):=\frac{\Res(\Phi_{n,\phi},p)}{\Res(\Phi_{n,\phi},q)}\cdot \ell^{\deg(q)-\deg(p)}.\]
Then the following hold:
\begin{enumerate}[(a)]
    \item\label{multiplier_product} $U_n(\phi)$ is the product of the multipliers of all the $n$-cycles of $\phi$.
    \item\label{Cn_criterion} $U_n(\phi)=0$ if and only if $\phi$ has an $n$-periodic critical point in $\PP^1(\bar K)$.
\end{enumerate}
\end{lem}

\begin{proof}
We write $\Phi_n$ for $\Phi_{n,\phi}$ throughout the proof. Note that \eqref{Cn_criterion} follows from \eqref{multiplier_product} using \eqref{Lambda_property}. By basic properties of polynomial resultants, we have
\[U_n(\phi)=\prod_{\Phi_n(\beta)=0}\frac{p(\beta)}{q(\beta)}=\prod_{\Phi_n(\beta)=0}\phi'(\beta).\]
If $\beta$ is any root of $\Phi_n$, the chain rule yields the identity
\[\lambda_{\phi,\beta}=\prod_{i=1}^{n}\phi'(\phi^i\beta).\]
The roots of $\Phi_n$ can be partitioned into $n$-cycles, and the above identity implies that roots within the same cycle have the same multiplier. Choosing representatives $\beta_1,\ldots, \beta_r$ of the $n$-cycles, part \eqref{multiplier_product} follows from the identities
\[\prod_{j=1}^r\lambda_{\phi,\beta_j}=\prod_{j=1}^r\prod_{i=1}^{n}\phi'(\phi^i\beta_j)=\prod_{\Phi_n(\beta)=0}\phi'(\beta)=U_n(\phi).\qedhere\]
\end{proof}

The following notation will be used throughout the remainder of the article: for elements $r,s\in K$, we define a rational function $\phi_{r,s}\in K(x)$ by
\begin{equation}\label{phi_rs}
\phi_{r,s}:=\frac{2x^2+(2-r)x+(2-r)}{-x^2+(2+r)x+2-r-s}.
\end{equation}

As observed by Manes and Yasufuku \cite{manes-yasufuku}*{Remark 3.2}, the map $\phi_{r,s}$ has degree two if and only if $r,s$ do not satisfy the equation \eqref{symlocuseq}, and in that case, $(r,s)$ are the coordinates of $\phi_{r,s}$ under the isomorphism \eqref{M2A2_iso}; in particular, $\Aut(\phi_{r,s})$ is trivial.   

\begin{prop}\label{Cn_equation}
Let $r,s$ be indeterminates, and $K=\Q(r,s)$. Set \[F_n:=U_n(\phi_{r,s})\in\Q(r,s),\]
where $U_n$ is defined as in Lemma \ref{Un_lem}. Then $F_n$ is a polynomial in $r$ and $s$ with integer coefficients. Moreover, as a subset of $\A^2(\Qbar)$, the set $C_n$ is equal to the set of points on the curve $F_n=0$. 
\end{prop}
\begin{proof}
By Lemma \ref{Un_lem}, $F_n$ is the product of the multipliers of the $n$-cycles of $\phi_{r,s}$. Thus, in the notation used by Milnor \cite{milnor}*{Corollary D.2}, we have $F_n=\sigma_N^{(n)}(\phi_{r,s})$, where $N$ is the number of $n$-cycles of $\phi_{r,s}$, and therefore $C_n$ (denoted $\Per_n(0)$ by Milnor) is the curve defined by $F_n=0$. Moreover, it follows from \cite{silverman_Md}*{Corollary 5.2} that $F_n\in\Z[r,s]$.
\end{proof}

With $F_n$ as in Proposition \ref{Cn_equation}, we compute (using resultants) that
\begin{align*}
F_1 &= r - 2,\\
F_2 &=2r + s,\\
F_3 &=2r^3 + 5r^2s - r^2 + 4rs^2 - 2rs + 12r + s^3 + 28,\\
F_4 &=2r^5+4r^4s^2-3r^4s+\cdots+60s^3-48s^2+96s+304.
\end{align*}
In particular, we recover the equations for $C_1$, $C_2$, and $C_3$ obtained by Milnor in \cite{milnor}*{Lemma 3.4 and Example D.3}. Using these equations, we will now parametrize the curves $C_2$, $C_3$, and $C_4$. Let $t$ be an indeterminate, and define rational functions $\varepsilon,\rho,\eta,\kappa\in\Q(t)$ by 

\begin{align}\label{etakappa}
\begin{split}
    \varepsilon &=-(t^3 - t^2 + 5t + 1)/t,\\
    \rho &= (t^3 - 2t^2 + 7t + 2)/t,\\
    \eta &=-\frac{t^6 + t^5 + 2t^4 + 3t^3 - 2t^2 - 5t - 1}{t(t - 1)(t + 1)^2},\\  
    \kappa &=\frac{t^5 + 2t^4 + 3t^3 - 5t - 2}{t(t^2 - 1)}.
\end{split}
\end{align}

\begin{prop}\label{C2C3C4_eqns} The sets $C_n(\Q)\subseteq\A^2(\Q)$, $n\in\{2,3,4\}$, are described by
    \begin{align*}
        C_2(\Q)&=\{(v,-2v):v\in\Q\},\\
        C_3(\Q)&=\{(\varepsilon(v),\rho(v)):v\in\Q\setminus\{0\}\},\\
        C_4(\Q)&=\{(\eta(v),\kappa(v)):v\in\Q\setminus\{0,\pm 1\}\}.
    \end{align*}
\end{prop}
\begin{proof}
The case of $C_2$ is trivial using the equation $2r+s=0$. The curves defined by $F_3=0$ and $F_4=0$ have genus 0, and can therefore be parametrized using the methods discussed in \cite{sendra-winkler-diaz}*{Chapters 4-6}. The \textsc{Magma} function \texttt{Parametrize} yields the stated descriptions of $C_3(\Q)$ and $C_4(\Q)$.
\end{proof}

Next, we will use the equation $F_n=0$ for the curve $C_n$ in order to describe the fields of moduli of dynamical systems in $C_n\cap\calS$. This, in particular, yields proofs of Theorems \ref{main_thm}\eqref{main_thm_nontrivial} and \ref{cv_extensions}\eqref{3_nontrivial}.

\begin{prop}\label{period234_autos} As subsets of $\calM_2(\Qbar)$, the sets $C_n\cap\calS$ satisfy the following:
\begin{enumerate}[(a)]
    \item\label{C2meetS} $C_2\cap\calS=\{[1/x^2]\}$;
    \item\label{C3meetS} $C_3\cap\calS$ has cardinality four, and the field of moduli of each of its elements is a number field of absolute degree four;
    \item\label{C4meetS} $C_4\cap\calS$ has cardinality eight, and the field of moduli of each of its elements has absolute degree two or six. 
\end{enumerate}
\end{prop}

\begin{proof}
For $n\in\{2,3,4\}$, we consider the coordinates of dynamical systems in the set $C_n\cap\calS$. Let $\A^2=\Spec\Q[r,s]$ and let $X_n$ be the subscheme of $\A^2$ defined by the equations $F_n(r,s)=0$ and \eqref{symlocuseq}. Note that $C_n\cap\calS$ is in bijection with $X_n(\Qbar)$ via the isomorphism \eqref{M2A2_iso}. A computation in \textsc{Magma} shows that $X_n$ is 0-dimensional and constructs all algebraic points on $X_n$. We find that $X_2$ has a single point, namely $(-6,12)$, and this point coincides with the coordinates of the dynamical system $[1/x^2]$. This proves \eqref{C2meetS}.

The fields of moduli referenced in parts \eqref{C3meetS} and \eqref{C4meetS} of the proposition are equal to the fields of definition of the points in $X_3(\Qbar)\cup X_4(\Qbar)$, and can thus be computed. For every point $P\in X_3(\Qbar)$ we find that $\Q(P)$ is isomorphic to the number field $\Q[x]/(f)$, where $f=x^4 - 16x^3 + 112x^2 - 320x + 512$. Similarly, we conclude that $X_4(\Qbar)$ has two points defined over the quadratic field $\Q(\sqrt{-3})$ and six points over the number field with defining polynomial $x^6 - 26x^5 + 259x^4 - 1248x^3 + 3328x^2 - 5312x + 43526$.
\end{proof}

\subsection{Rational conjugacy classes} Let $\Cn$ denote the set of all arithmetic dynamical systems $\langle\phi\rangle$, where $\phi\in\Rat_2(\Q)$ has an $n$-periodic critical point in $\PP^1(\Qbar)$. Equivalently,
\[\Cn=\{\langle\phi\rangle:\phi\in\Rat_2(\Q),\,[\phi]\in C_n\}.\]
Using the descriptions of the sets $C_n(\Q)$ and $C_n\cap\calS$ obtained 
in \textsection\ref{Cn_eqn_section}, we now derive explicit one-parameter descriptions of the sets $\Cn$.

\begin{lem}\label{manes-yasufuku_lem}
Suppose that $\phi\in\Rat_2(\Q)$ has trivial automorphism group. Let $(r,s)\in\A^2(\Qbar)$ be the coordinates of $[\phi]\in\calM_2(\Qbar)$, and define $\phi_{r,s}$ as in \eqref{phi_rs}. Then $r,s\in\Q$ and $\langle\phi\rangle=\langle\phi_{r,s}\rangle$. 
\end{lem}
\begin{proof}
 The monic polynomial whose roots are the elements of $\Lambda_1(\phi)$ has rational coefficients; see \cite{silverman_book}*{Theorem 4.50}. In particular, the coordinates $(r,s)$ are rational. Since $\Aut(\phi)$ is trivial, the point $(r,s)$ does not lie on the curve $\calS$, and therefore  $\phi_{r,s}\in\Rat_2(\Q)$. Moreover, we have $[\phi]=[\phi_{r,s}]$ by \cite{manes-yasufuku}*{Lemma 3.1}, and $\langle\phi\rangle=\langle\phi_{r,s}\rangle$ by \cite{silverman_book}*{Proposition 4.73}.
\end{proof}

\begin{prop}\label{parametrization2}
Define $\phi_c$ for $c\in\Q\setminus\{0\}$ and $\psi_c$ for $c\in\Q\setminus\{0,4\}$ by
\[\phi_c=\frac{c(x-1)}{x^2}\quad\text{and}\quad\psi_c=\frac{2x-1}{cx^2-1}.\]
\begin{enumerate}[(a)]
    \item\label{C2partition} The set $\langle C_2\rangle$ is a disjoint union
    \[\{\langle\phi_c\rangle:c\in\Q\setminus\{0\}\}\cup\{\langle\psi_c\rangle:c\in\Q\setminus\{0,4\}\},\]
    where the left-hand set corresponds to maps $\phi$ such that $\Aut(\phi)$ is trivial, and the right-hand set to those with $\Aut(\phi)$ nontrivial.
    \item\label{phi_v_critical} Every map $\phi_c$ has exactly one $2$-periodic critical point, namely $0$.
    \item\label{psi_v_critical} Every map $\psi_c$ has two $2$-periodic critical points, namely
    \[\alpha=\frac{c+\sqrt{c(c-4)}}{2c}\quad\text{and}\quad 1-\alpha=\frac{c-\sqrt{c(c-4)}}{2c}.\] Moreover, these points form a $2$-cycle under iteration of $\psi_c$.
\end{enumerate}
\end{prop}

\begin{proof}
For every nonzero $c\in\Q$, we compute that the coordinates of $[\phi_c]$ are $(r,s)=(c-6,12-2c)$. Since $s=-2r$, Proposition \ref{C2C3C4_eqns} implies that $[\phi_c]\in C_2$ and thus $\langle\phi_c\rangle\in\langle C_2\rangle$. Moreover, the left-hand side of the equation \eqref{symlocuseq} takes the value $c^2$, which is nonzero, so $\Aut(\phi_c)$ is trivial.

Similarly, the coordinates of every map $\psi_c$ are $(r,s)=(-6,12)$, which do satisfy \eqref{symlocuseq}, so $\Aut(\psi_c)$ is nontrivial. Moreover, we have $s=-2r$, so $[\psi_c]\in C_2$ and therefore $\langle\psi_c\rangle\in\langle C_2\rangle$. This proves one inclusion in \eqref{C2partition}; the reverse inclusion follows from Lemmas 2.5 and 2.6 in \cite{krumm-lacy}.

To prove \eqref{phi_v_critical}, we compute that the two critical points of $\phi_c$ are $0$ and $2$, and $0$ is 2-periodic for $\phi_c$. Moreover, $2$ is not 2-periodic, as $\phi_c^{2}(2)=(4c - 16)/c$, which equals $2$ if and only if $c=8$, in which case $2$ is fixed by $\phi_c$. This proves \eqref{phi_v_critical}, and part \eqref{psi_v_critical} is easily verified.
\end{proof}

\begin{prop}\label{parametrization3}
Defining maps $\phi_c\in\Q(x)$ for $c\in\Q\setminus\{0\}$ by
\[\phi_c=1-\frac{c+1}{x}+\frac{c}{x^2},\]
we have $\langle C_3\rangle=\{\langle\phi_c\rangle:c\in\Q\setminus\{0\}\}$. Moreover, every map $\phi_c$ has either one or two $3$-periodic critical points, which in any case are defined over $\Q$.
\end{prop}

\begin{proof} Suppose that $\phi\in\Rat_2(\Q)$ has a 3-periodic critical point, so that $\langle\phi\rangle\in\langle C_3\rangle$. Since $\Q_{\phi}=\Q$, Lemma \ref{period234_autos}\eqref{C3meetS} implies that $\Aut(\phi)$ is trivial. Letting $(r,s)\in C_3(\Q)$ be the coordinates of $[\phi]$, we have $\langle\phi\rangle=\langle\phi_{r,s}\rangle$ by Lemma \ref{manes-yasufuku_lem}. Proposition \ref{C2C3C4_eqns} implies that $(r,s)=(\varepsilon(v),\rho(v))$ for some nonzero $v\in\Q\setminus\{0\}$. The point $p:=-(v^2+3)/(v+1)$ is then a (rational) 3-periodic critical point of $\phi_{r,s}$. It follows that we can conjugate $\phi_{r,s}$ by an element of $\PGL_2(\Q)$ so that the critical 3-cycle becomes $0\mapsto\infty\mapsto 1$, with $0$ as critical point. We thus obtain the map $\phi_c$ with $c=(v^2+1)/(v^3+v)$. Since $\langle\phi\rangle=\langle\phi_{r,s}\rangle=\langle\phi_c\rangle$, this proves the inclusion $\langle C_3\rangle\subseteq\{\langle\phi_c\rangle:c\in\Q\setminus\{0\}\}$. 

In order to prove the reverse inclusion it suffices to prove the second statement of the proposition. The statement holds if $c=\pm1$, as the critical points of $\phi_c$ are then $0$ and $1$, or $0$ and $\infty$. Assuming now that $c\in\Q\setminus\{0,\pm1\}$, the critical points of $\phi_c$ are $0$ and $2c/(c+1)$, and $0$ is $3$-periodic for $\phi_c$.
\end{proof}

\begin{prop}\label{parametrization4}
Defining maps $\phi_c\in\Q(x)$ for $c\in\Q\setminus\{0,\pm 1\}$ by
\begin{equation}\label{C4_parametrized}
    \phi_c=\frac{(c + c^2 - c^3)x - c^2}{(c^3 - c^2 - c + 1)x^2 - (c^3 - c^2 - c)x - c^2},
\end{equation}
we have $\langle C_4\rangle=\{\langle\phi_c\rangle:c\in\Q\setminus\{0,\pm 1\}\}$. Moreover, every map $\phi_c$ has exactly one $4$-periodic critical point, namely $0$.
\end{prop}

\begin{proof}
Suppose that $\phi\in\Rat_2(\Q)$ has a $4$-periodic critical point, so that $\langle\phi\rangle\in\langle C_4\rangle$. Then $\Aut(\phi)$ is trivial by Lemma \ref{period234_autos}\eqref{C4meetS}. Letting $(r,s)\in C_4(\Q)$ be the coordinates of $[\phi]$, we have $\langle\phi\rangle=\langle\phi_{r,s}\rangle$ by Lemma \ref{manes-yasufuku_lem}; moreover, Proposition \ref{C2C3C4_eqns} implies that $(r,s)=(\eta(v),\kappa(v))$ for some $v\in\Q\setminus\{0,\pm1\}$. 
The point $p:=(v^4-v^2+2v+3)/(v^3-2v-1)$ is a $4$-periodic critical point of $\phi_{r,s}$, so we can conjugate $\phi_{r,s}$ over $\Q$ in order to have critical cycle $0\mapsto1\mapsto-v\mapsto\bullet$, with $0$ as critical point, thus obtaining the map $\phi_c$ with $c=-v$. Since ${\langle\phi\rangle=\langle\phi_c\rangle}$, this proves that $\langle C_4\rangle\subseteq\{\langle\phi_c\rangle:c\in\Q\setminus\{0,\pm1\}\}$. The reverse inclusion follows from the second statement of the proposition, which in turn is proved by a simple calculation: the critical points of any map $\phi_c$ are $0$ and $-2c/(c^2 - c - 1)$, and the latter is not $4$-periodic for $\phi_c$.
\end{proof}
\begin{rem}
    The parametrization of $\langle C_4\rangle$ given in Proposition \ref{parametrization4} was chosen to have the property that $\infty$ is not periodic for the maps $\phi_c$, as this is helpful for some of the computations in Section \ref{period4_section}. However, disregarding this requirement, a simpler parametrization is given by maps of the form $\phi_c=[cx^2+(c^2-3c+1)x-(2c^2-3c+1)]/(cx^2)$ with $c\in\Q\setminus\{0,1/2,1\}$.
\end{rem}

\section{Portraits with a critical cycle of length $2$ or $3$}\label{period2_3_section}
In this section we prove Theorem \ref{cv_extensions}, which concerns the portraits of quadratic rational maps with a 2-periodic or 3-periodic critical point. As noted earlier, part \eqref{3_nontrivial} of the theorem follows immediately from Proposition \ref{period234_autos}\eqref{C3meetS}. Propositions \ref{period3_portraits}, \ref{period2_trivial_portraits}, and \ref{period2_autos_portraits} below prove parts \eqref{3_trivial}, \eqref{2_trivial}, and \eqref{2_nontrivial} of the theorem, respectively.

\begin{prop}\label{period3_portraits}
Suppose that $\phi\in\Rat_2(\Q)$ has a $3$-periodic critical point in $\PP^1(\Qbar)$. Then $\phi$ has such a point defined over $\Q$. In particular, under the assumptions of Theorem 1 in \cite{vishkautsan}, the portrait $G(\phi,\Q)$ must be one of the six referred to in the theorem.
\end{prop}
\begin{proof}
By Proposition \ref{parametrization3}, $\phi$ must have a rational $3$-periodic critical point. The result now follows from  \cite{vishkautsan}*{Theorem 1}.
\end{proof}

\begin{prop}\label{period2_trivial_portraits}
 Suppose that $\phi\in\Rat_2(\Q)$ has a $2$-periodic critical point in $\PP^1(\Qbar)$, and that $\Aut(\phi)$ is trivial. Then $\phi$ has a unique $2$-periodic critical point, and this point is defined over $\Q$. In particular, under the assumptions of Theorem 1.2 in \cite{canci-vishkautsan}, the portrait $G(\phi,\Q)$ must be one of the 13 referred to in the theorem.
\end{prop}
\begin{proof}
The first statement follows from Proposition \ref{parametrization2}, and the second follows from \cite{canci-vishkautsan}*{Theorem 1.2}.
\end{proof}

\begin{prop}\label{period2_autos_portraits}
Suppose $\phi\in\Rat_2(\Q)$ has nontrivial automorphism group and a $2$-periodic critical point. Then $\phi$ has exactly two such points, $p$ and $\phi(p)$, and these points generate a number field $K$ of absolute degree at most two. Moreover, the portrait $G(\phi,K)$ must be one of the four shown in Figure \ref{period2portraits}. Conversely, every portrait in the figure is realizable in the form $G(\phi,K)$.
\end{prop}
\begin{proof}
By Proposition \ref{parametrization2}, it suffices to prove the result when $\phi$ is a map of the form $\psi_c=(2x-1)/(cx^2-1)$. Recall that $\psi_c$ has two critical points, $\alpha$ and $1-\alpha$. Hence we can conjugate $\psi_c$ over $K=\Q(\alpha)$ in order to have critical cycle $0\mapsto\infty$, obtaining the map
\[\widetilde\psi_c=\frac{\alpha}{1-\alpha}\cdot\frac{1}{x^2}=\frac{c\alpha-1}{x^2}.\]
To determine the possible portraits of this map (which has the same portrait as $\phi$ over $K$), we consider two cases.
    
Suppose first that $c\alpha-1$ does not have a cube root in $\Q(\alpha)$. Then $\widetilde\psi_c$ has no pre-periodic points other than $0$ and $\infty$, since every such point is of the form $\zeta(c\alpha-1)^{1/3}$ for some root of unity $\zeta$. Hence, the portrait $G(\widetilde\psi_c,\Q(\alpha))$ must be the portrait on the top left of Figure \ref{period2portraits}.

Now suppose that $c\alpha-1$ has a cube root $\omega\in\Q(\alpha)$. In this case, $\widetilde\psi_c$ is conjugate to $1/x^2$ via the map $x\mapsto x/\omega$, so the portrait of $\widetilde\psi_c$ is the same as the portrait of $1/x^2$ over $\Q(\alpha)$. Since $\Q(\alpha)$ is either $\Q$ or a quadratic number field, the analysis in \cite{lukas-manes-yap}*{Section 5.2} shows that there are three possibilities for this portrait, namely the remaining three in Figure \ref{period2portraits}.

Finally, we realize each portrait from the figure in the form $G(\psi_c,\Q(\alpha))$. The first portrait shows the typical case, when $c\alpha-1$ is not a cube in $\Q(\alpha)$. Choosing $c=81/8,2,$ and $400/343$, respectively, we obtain the second, third, and fourth portraits in the figure. The chosen values of $c$ are obtained by setting $q=2$, $p=1$, and $p=7/9$ in the following lemma, a proof of which is given in Appendix \ref{realization_appendix}.

\begin{lem}\label{period2_realizations}
With notation as in Proposition \ref{parametrization2}\eqref{psi_v_critical}, the following hold.
\begin{enumerate}[(a)]
    \item If $c=(q^3+1)^2/q^3$ for some $q\in\Q\setminus\{0,\pm 1\}$, then $\alpha$ is rational and $c\alpha-1$ has a rational cube root. The converse holds as well.
    \item If $\alpha$ is irrational and $c\alpha-1$ has a cube root in $\Q(\alpha)$, then there exists $p\in\Q\setminus\{0,\pm 1,1/2,1/3\}$ such that $c=(4p^3-3p+1)/p^3$. The converse holds if and only if $(p+1)(1-3p)$ is not a rational square.\qedhere
\end{enumerate}
\end{lem}
\end{proof}

\begin{figure}
\centering
\includegraphics[scale=0.45]{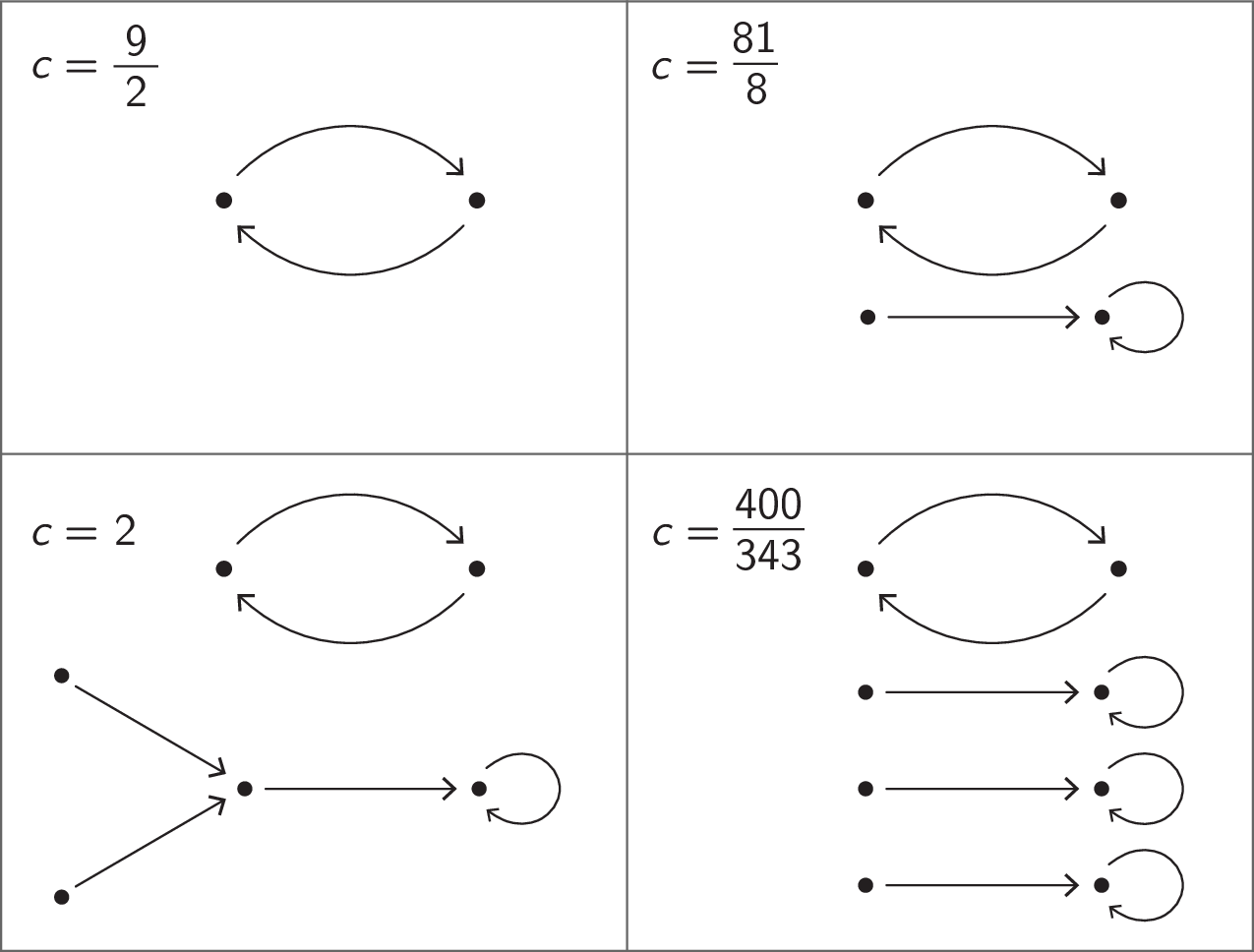}
\caption{Possibilities for the portrait of $\psi_c=\frac{2x-1}{cx^2-1}$ over the field generated by its critical points.}
\label{period2portraits}
\end{figure}

\section{Dynamical modular curves}\label{dmc_section}
In this section we define several types of algebraic curves that will be used in our analysis of the dynamics of the maps $\phi_c$ in \eqref{C4_parametrized}. A similar analysis could, in principle, be applied to study the dynamics of any family of maps $\phi_c:=\phi(c,x)$, where $\phi\in\Q(t)(x)$ is a nonconstant rational function with coefficients in a rational function field over $\Q$.

All the curves used here will be defined in terms of the following construction. Let $t$ and $x$ be indeterminates, and $f(t,x)\in\Q(t)(x)$ a nonconstant rational function in $x$ with coefficients in the field $\Q(t)$. Let $p\in\Q(t)$ be the content of $f$, so that we may express $f$ uniquely in the form
\begin{equation}\label{normalization}
    f(t,x)=p(t)\cdot\frac{A(t,x)}{B(t,x)},
\end{equation} where $A,B\in\Q[t][x]$ are coprime polynomials, each with content equal to 1. If $A$ is nonconstant, we denote by $Z_f$ the affine plane curve defined by the equation $A(t,x)=0$. The \textbf{exceptional set} of $f$, denoted $\calE(f)$, is defined to be the finite subset of $\Qbar$ consisting of the zeros and poles of $p$, and the $t$-coordinates of all algebraic points on the subscheme of $\A^2=\Spec\Q[t,x]$ defined by $A(t,x)$ and $B(t,x)$. We then have the following equivalence for any point $(t_0,x_0)\in\A^2(\Qbar)$ with $t_0\notin\calE(f)$:
\begin{equation}\label{Zf_ppty}
    f(t_0,x_0)=0\Longleftrightarrow(t_0,x_0)\in Z_f(\Qbar).
\end{equation}

It follows, in particular, that if the curve $Z_f$ has only finitely many \emph{rational} points, all of which are explicitly known, then all rational solutions of the equation $f(t_0,x_0)=0$ can be effectively determined. Note that if $f\in\Q(t)[x]$ is a polynomial, the representation \eqref{normalization} will have $B=1$.

Now fix a nonconstant rational function $\phi\in\Q(t)(x)$ giving rise to a finite morphism $\phi:\PP^1_{\Q(t)}\to\PP^1_{\Q(t)}$. By base change, we may regard $\phi$ as a map $\PP^1_{\Qbar(t)}\to\PP^1_{\Qbar(t)}$. For all but finitely many $c\in\Qbar$, a process of reduction modulo the ideal $(t-c)$ yields a specialized map $\phi_c:\PP^1\to\PP^1$ which is defined over the number field $\Q(c)$ and has the same degree as $\phi$; see \cite{silverman_book}*{\textsection\textsection 2.3-2.5}. Explicitly, write $\phi(t,x)=p(t)\cdot\frac{A(t,x)}{B(t,x)}$ as in \eqref{normalization}. For the purposes of this article, we define the \textbf{bad set} of $\phi$, denoted $\calB(\phi)$, to be the finite set of elements of $\overline\Q$ that are either zeros or poles of $p$; roots of the leading coefficients of $A$ or $B$; or roots of the resultant $\Res(A,B)\in\Q[t]$. Then, for $c\in \Qbar\setminus\calB(\phi)$, the morphism $\phi_c:\PP^1\to\PP^1$ induced by the rational function $\phi(c,x)$ has the same degree as $\phi$, and is defined over $\Q(c)$. Applying the $Z_f$-construction, we will now attach to $\phi$ various \emph{dynamical modular curves} whose point sets encode dynamical information about the maps $\phi_c$.

For every positive integer $n$, the $n$th \textbf{dynatomic curve} of $\phi$, denoted $Y_1(n)$ when $\phi$ is clear from context, is the curve $Z_f$ corresponding to the dynatomic polynomial $\Phi_n=\Phi_{n,\phi}\in\Q(t)[x]$. Assuming that $\Phi_n$ has nonzero discriminant $\Delta_n\in\Q(t)$, it follows from \eqref{Zf_ppty} that a generic point on $Y_1(n)$ corresponds to a pair $(c,\alpha)$ such that $\alpha$ is $n$-periodic for the map $\phi_c$. More explicitly, suppose that $c,\alpha\in\Qbar$ are such that $c\notin\calB(\phi)\cup\calE(\Phi_n)$ and $\alpha$ is $n$-periodic for the map $\phi_c$. Then $\Phi_n(c,\alpha)=0$ and thus $(c,\alpha)$ is a point on $Y_1(n)$. Conversely, if $(c,\alpha)\in Y_1(n)(\Qbar)$ is any point with $c\notin\calB(\phi)\cup\calE(\Phi_n)$ and such that $c$ is neither a zero nor a pole of $\Delta_n$, then $\alpha$ is $n$-periodic for $\phi_c$. Note that if $\infty$ is not $n$-periodic for any map $\phi_c$ with $c\in\Qbar$, then every pair $(c,P)$, where $c\notin\calB(\phi)\cup\calE(\Phi_n)$ and $P\in\PP^1(\Qbar)$ is an $n$-periodic point of $\phi_c$, represents a point on $Y_1(n)$; moreover, there is a rational map $\rho:Y_1(n)\to Y_1(n)$ given by $(t,x)\mapsto (t,\phi(t,x))$.

For positive integers $m$ and $n$, we define a \textbf{generalized dynatomic polynomial} of $\phi$ by the equation
\[\Phi_{m,n}:=(q_m/q_{m-1})^{D(n)}\cdot\frac{\Phi_n(\phi^m)}{\Phi_n(\phi^{m-1})},\]
where $D(n)$ is the degree of $\Phi_n$ and, for each index $i$, $q_i\in\Q(t)[x]$ is a denominator for the rational function $\phi^i$, so that $\phi^i=p_i/q_i$ with $p_i$ and $q_i$ being coprime polynomials. The polynomial $\Phi_{m,n}$ has the property that its set of roots includes all elements of $\overline{\Q(t)}$ having pre-periodic type $(m,n)$ for $\phi$ (see \cite{hutz}*{Theorem 1} and \cite{silverman_book}*{Exercise 4.11}), and the specialized polynomial $\Phi_{m,n}(c,x)$ has the same property relative to the map $\phi_c$. The curve $Z_f$ corresponding to $f=\Phi_{m,n}$ will be denoted $Y_1(m,n)$ and called a \textbf{generalized dynatomic curve} of $\phi$. 

Under mild assumptions, a generic point on $Y_1(m,n)$ corresponds to a pair $(c,\alpha)$ such that $\alpha$ has pre-periodic type $(m,n)$ for the map $\phi_c$. Indeed, if $c,\alpha\in\Qbar$ are such that $c\notin\calB(\phi)\cup\calE(\Phi_{m,n})$ and $\alpha$ has pre-periodic type $(m,n)$ for $\phi_c$, then $(c,\alpha)$ is a point on $Y_1(m,n)$. Conversely, suppose that $\Delta_n$ is nonzero; set $\lambda_i:=\Res_x(\Phi_{m,n},\Phi_{i,n})\in\Q(t)$ for $1\le i< m$; and suppose every $\lambda_i$ is nonzero. If $(c,\alpha)\in Y_1(m,n)(\Qbar)$ is any point satisfying all of the following:
\begin{itemize}
\item $c\notin\calB(\phi)\cup\calE(\Phi_{m,n})$;
\item $c$ is neither a zero nor a pole of $\Delta_n$; 
\item $c$ is neither a zero nor a pole of $\lambda_i$ for any $i<m$;
\end{itemize}
then $\alpha$ has pre-periodic type $(m,n)$ for $\phi_c$.

Note that, for every $1\le i<m$, there is a rational map $Y_1(m,n)\to Y_1(i,n)$ given by $(t,x)\mapsto(t,\phi^{m-i}(t,x))$.

To define the next collection of curves, let $n$ be a positive integer; we assume that $\Phi_n$ has nonzero discriminant and that $\infty$ is not $n$-periodic for $\phi$. The $n$th \textbf{trace polynomial} of $\phi$, denoted $T_{n,\phi}$, is the unique monic polynomial in $\Q(t)[x]$ whose $n$th power is given by
\[T_{n,\phi}(x)^n=\prod_{\Phi_{n}(\beta)=0}(x-\tau(\beta)),\] where $\tau:=\sum_{i=0}^{n-1}\phi^i\in\Q(t)(x)$. The curve $Z_f$ corresponding to $T_{n,\phi}$ will be denoted by $Y_{\tau}(n)$ and called the $n$th \textbf{trace curve} of $\phi$. The significance of this curve lies in the existence of a rational map $Y_1(n)\to Y_{\tau}(n)$ given by $(t,x)\mapsto(t,\tau(t,x))$, which can be useful, for instance, when attempting to compute the set of rational points on $Y_1(n)$.

Though it will not be needed here, we briefly mention a related construction. If the natural maps $\rho:Y_1(n)\to Y_1(n)$ and $Y_1(n)\to Y_{\tau}(n)$ are morphisms (e.g., if $\phi$ is a polynomial), then $\rho$ is an automorphism of order $n$ and we may consider the quotient $Y_0(n):=Y_1(n)/\langle\rho\rangle$. Moreover, the map $Y_1(n)\to Y_{\tau}(n)$ factors through $Y_0(n)$, and the induced map $Y_0(n)\to Y_{\tau}(n)$ can be shown to be an isomorphism if both $\Phi_{n,\phi}$ and $T_{n,\phi}$ are irreducible.

The following lemma can be used to compute the $n$th power of the trace polynomial $T_{n,\phi}$, specifically by applying the lemma to the data $K=\Q(t)$, $p=\Phi_n$, and $f(x)=\sum_{i=0}^{n-1}\phi^i(x)$. A proof of the lemma is omitted, as it is a straightforward calculation based on properties of resultants.
\begin{lem}\label{trace_lem}
Let $K$ be a field and $p\in K[x]$ a polynomial of degree $m\ge 1$ with roots $\alpha_1,\ldots,\alpha_m\in\bar K$ and leading coefficient $\ell$. Let $f=q/h\in K(x)$ be nonconstant, where $q,h\in K[x]$, and let $e=\max(\deg h,\deg q)-\deg h$. If $p$ and $h$ have no common root in $\bar K$, then
\[\prod_{i=1}^m(x-f(\alpha_i))=\frac{\Res_y(p(y),h(y)x-q(y))}{\Res(p,h)\cdot\ell^e}.\]
\end{lem}

Finally, we define a family of \textbf{preimage curves} attached to the rational function $\phi$. For $m\ge 1$ and $P\in\PP^1(\Q(t))$ we define $Y(m,P)$ to be the curve $Z_f$ corresponding to the rational function
\[f=\begin{cases}
    \phi^m-P &\text{if }P\ne\infty,\\
    1/\phi^m & \text{if }P=\infty.
\end{cases}\] When convenient, we will refer to $Y(m,P)$ as ``the curve $\phi^m=P$." Note that the curves $Y(m,\infty)$ are not defined in the case where $\phi$ is a polynomial, as the equation $1/\phi^m=0$ does not then define a curve.

 For $c\in\Qbar$ and $P$ as above, let $P(c)\in\PP^1(\Qbar)$ be the reduction of $P$ modulo $(t-c)$ as defined in \cite{silverman_book}*{\textsection 2.3}; thus, $\infty(c)=\infty$, and $P(c)$ is the value of $P$ at $c$ if $P\in\Q(t)$. A generic point on $Y(m,P)$ then corresponds to a pair $(c,\alpha)$ such that $\phi_c^m(\alpha)=P(c)$. To verify this in the case $P\ne\infty$, suppose $c,\alpha\in\Qbar$ satisfy $c\notin\calB(\phi)\cup\calE(f)$ and $c$ is not a pole of $P$. Then the definitions imply that $(c,\alpha)\in Y(m,P)(\Qbar)\iff\phi_c^m(\alpha)=P(c)$.
 
 Note that, for $1\le i\le m$, there is a rational map $Y(m,P)\to Y(i,P)$ given by $(t,x)\mapsto(t,\phi^{m-i}(t,x)))$.

All of the dynamical modular curves defined above have a natural map onto $\A^1$, namely the projection onto the first coordinate, which we call the $\bm{c}$\textbf{-map}. These curves may thus be regarded as schemes over $\A^1$, and all the natural maps between these modular curves, such as the map $Y_1(n)\to Y_{\tau}(n)$, are morphisms of $\A^1$-schemes. Moreover, fiber products of such schemes can be used to describe the specialized maps $\phi_c$ having a combination of dynamical properties. For example, if $\phi_c$ has rational points of periods $m$ and $n$, then $c$ is the projection of a rational point on the curve $Y_1(m)\times_{\A^1}Y_1(n)$. Note that if $A_m(t,x)=0$ and $A_n(t,x)=0$ are the defining equations of $Y_1(m)$ and $Y_1(n)$, respectively, then a model for $Y_1(m)\times_{\A^1}Y_1(n)$ in $\A^3=\Spec\Q[t,x,z]$ is given by the pair of equations $A_m(t,x)=0$ and
$A_n(t,z)=0$.

Having defined all the required modular curves, we now apply these constructions to study the rational portraits of the maps $\phi_c$ in \eqref{C4_parametrized}.

\section{Portraits with a critical $4$-cycle}\label{period4_section}
In this section we prove Theorem \ref{main_thm}, the main result of this article, which concerns the portraits of quadratic rational maps with a 4-periodic critical point. As noted earlier, part \eqref{main_thm_nontrivial} of the theorem follows immediately from Proposition \ref{period234_autos}\eqref{C4meetS}. The first statement of part \eqref{main_thm_trivial} follows from Proposition \ref{parametrization4}, and the remainder is proved in Proposition \ref{period4_finiteness} below.

\subsection{Search for portraits} Before the proof of Theorem \ref{main_thm}, we summarize the results of a computation carried out in order to search for portraits of dynamical systems in $\langle C_4\rangle$. By Proposition \ref{parametrization4}, if $\langle\phi\rangle\in\langle C_4\rangle$, then $G(\phi,\Q)=G(\phi_c,\Q)$ for some $c\in\Q\setminus\{0,\pm 1\}$, where $\phi_c$ is defined by \eqref{C4_parametrized}. This suggests choosing a sample set of rational numbers $c$ and computing the portraits $G(\phi_c,\Q)$ using an algorithm of Hutz \cite{hutz} that is included in \textsc{Sage}. Recall that the \emph{height} of $c\in\Q$ is given by $H(c)=\max\{|a|,|b|\}$ if $c=a/b$ with $a,b\in\Z$ and $\gcd(a,b)=1$. 

\begin{prop}\label{portrait_search}
    For every rational number $c\notin\{0,\pm 1\}$ with $H(c)\le 500$, the portrait $G(\phi_c,\Q)$ is one of the five shown in Figure \ref{period4portraits_trivial}.
\end{prop}

Based on the computation described above (see also Remark \ref{period4_rem}), we conjecture that the portraits in Figure \ref{period4portraits_trivial} represent \emph{all} portraits of the form $G(\phi_c,\Q)$. Our main goal in this section is to prove a number of statements supporting this conjecture, the main result being Proposition \ref{period4_finiteness}.

\subsection{Proof of Theorem \ref{main_thm}}\label{main_thm_section}
Define $\phi\in\Q(t)(x)$ by
\begin{equation}\label{period4_functionfield}
    \phi=\frac{(t + t^2 - t^3)x - t^2}{(t^3 - t^2 - t + 1)x^2 - (t^3 - t^2 - t)x - t^2}.
\end{equation}
Computing the set $\calB(\phi)\cap\Q$ we obtain $\{0,\pm 1\}$; thus, for $c\in\Q\setminus\{0,\pm 1\}$, the specialization $\phi_c$ is the map \eqref{C4_parametrized}.

We claim that the portrait of $\phi$ over $\Q(t)$ contains a copy of the portrait labeled I1 in Figure \ref{period4portraits_trivial}. Defining $A,B,Q\in\Q(t)$ by
\begin{equation}\label{ABQ_def}
    A:=1/(1-t),\quad B:=-t/(t^2 - t - 1),\quad Q:=t/(1-t^2),
\end{equation}
we compute that the orbit of $0$ under $\phi$ is the 4-cycle $0\mapsto 1\mapsto t\mapsto B$, and the non-periodic preimages of points in the cycle are $\infty,A,Q$, which map, respectively, to $0,B,t$. The claim follows, as illustrated in Figure \ref{I1_figure}. A further calculation shows that reduction modulo $(t-c)$ preserves the structure of this copy of I1, so that the portrait $G(\phi_c,\Q)$ also contains I1: indeed, $0$ is 4-periodic for $\phi_c$ by Proposition \ref{parametrization4}, and one can easily verify that the points $A(c)$ and $Q(c)$ do not belong to the orbit of $0$.

\begin{figure}[b]
    \centering
    \includegraphics[scale=0.19]{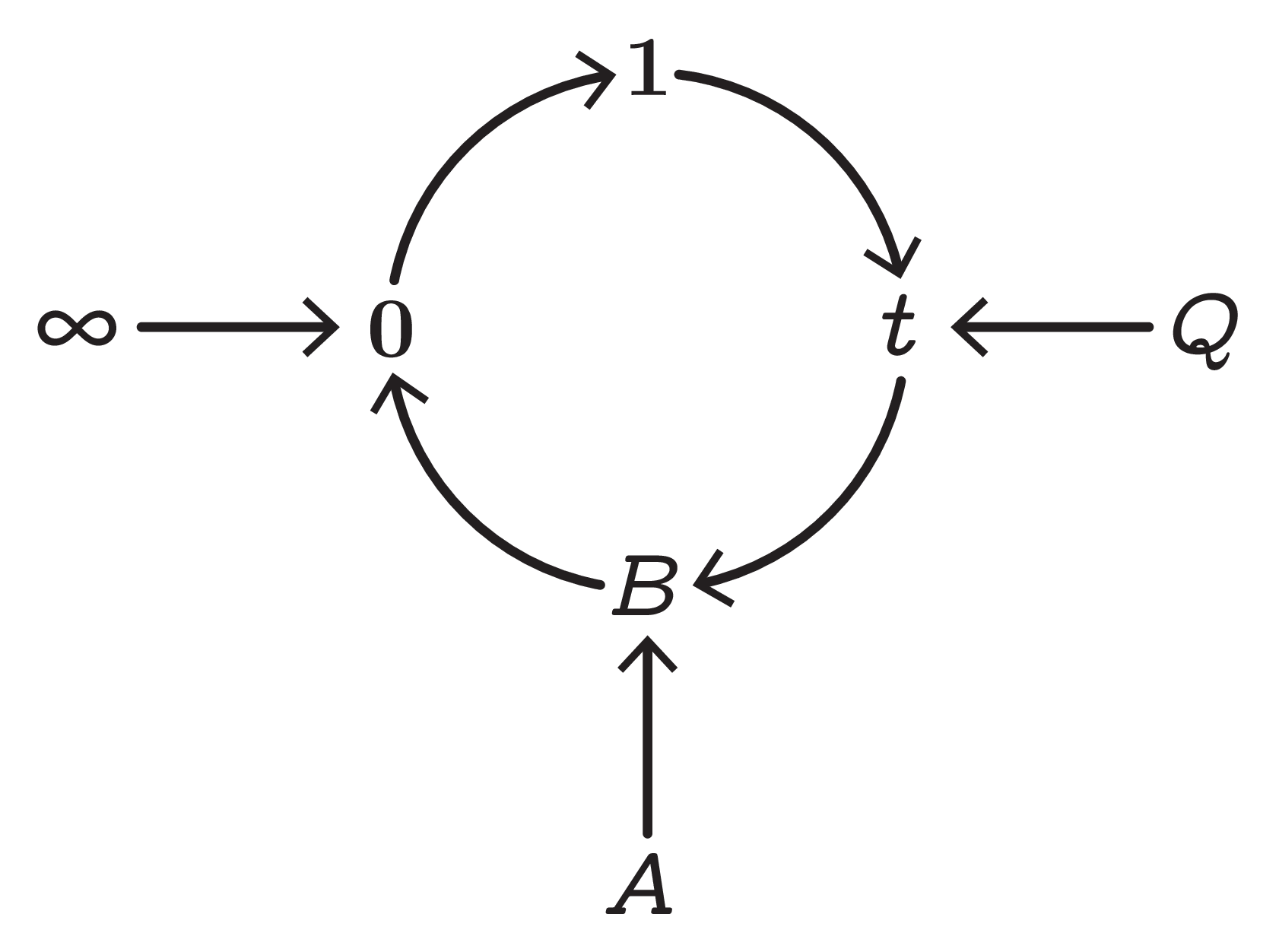}
    \caption{Portrait I1 within $G(\phi,\Q(t))$, where $\phi$ is given by \eqref{period4_functionfield}.}
    \label{I1_figure}
\end{figure}

A computation in \textsc{Magma} shows that the curve $Y_1(2,2)$ has genus 6, $Y(2,\infty)$ has genus $5$, and the fiber product
\[F:=Y_1(2)\times_{\A^1}Y(1,\infty)\]
has genus $5$. A search for rational points suggests that
\begin{equation}\label{diophantine}\tag{$\ast$}
  \#Y_1(2,2)(\Q)=5,\;\#Y(2,\infty)=3,\;\text{and }\#F(\Q)=9,
\end{equation}
though we are unable to prove these equalities. All three of the above curves are geometrically irreducible and not geometrically hyperelliptic; moreover, attempts to find nontrivial quotients of these curves failed due to limitations of computing time and/or memory. The natural maps $F\to Y(1,\infty)$ and $Y(2,\infty)\to Y(1,\infty)$ are not useful for finding all rational points of their domains, as the curve $Y(1,\infty)$ is birational to an elliptic curve of rank $1$.

As noted above, every portrait $G(\phi_c,\Q)$ contains a copy of I1. The following proposition will be used to deduce additional information about the possible shapes of these portraits.

\begin{prop}\label{preimage_curves} For $c\in\Q\setminus\{0,\pm 1\}$, the map $\phi_c$ satisfies the following:
\begin{enumerate}[(a)]
    \item \label{no_fixed_period3} $\phi_c$ has a rational fixed point if and only if $c=1/6$. Moreover, $\phi_c$ has no rational $3$-periodic point.
    \item \label{one_2cycle} If $\phi_c$ has a rational $2$-periodic point, then it has exactly two such points.
    \item \label{one_fourcycle} $\phi_c$ has exactly one $4$-cycle consisting of rational points.
    \item\label{AQ1_preimage} With notation as in \eqref{ABQ_def}, the points $A(c)$ and $Q(c)$ (which map into the critical $4$-cycle of $\phi_c$) do not have rational preimages under $\phi_c$. Moreover, $1$ has no rational preimage other than $0$.
    \item \label{2234} Assuming \eqref{diophantine}, $\phi_c$ does not have a rational point of pre-periodic type $(2,2)$; moreover, $\phi_c$ has both a rational $2$-periodic point and a pre-periodic rational point of type $(2,4)$ if and only if $c=5/2$.
\end{enumerate} 
\end{prop}

\begin{proof}
The curve $Y_1(1)$ has genus 2 and its Jacobian has Mordell--Weil rank 1; the \textsc{Magma} function \texttt{Chabauty} can thus determine its set of rational points. We find that every rational point on $Y_1(1)$ has $c$-coordinate in the set $\{0,1/6,\pm 1\}$. Thus, $\phi_c$ has a rational fixed point if and only if $c=1/6$.

The curve $Y_{\tau}(3)$ is birational to an elliptic curve of Mordell--Weil rank 0, so its rational point set, and hence that of $Y_1(3)$, is finite and computable. We find that every rational point on $Y_1(3)$ has $c\in\{0,\pm 1\}$,  proving \eqref{no_fixed_period3}. Part \eqref{one_2cycle} follows from the fact that the second dynatomic polynomial of $\phi_c$ is a quadratic polynomial with nonzero discriminant. 

The curve $Y_1(4)$ has four rational irreducible components (corresponding to the points in the critical 4-cycle of $\phi$) and one additional component of genus 24, denoted by $U$. If $\phi_c$ has a rational 4-cycle distinct from the cycle contained in the copy of I1, then $c$ is the projection of a rational point on $U$. Thus, it suffices to show that every rational point on $U$ has $c$-coordinate in the set $\{0,\pm 1\}$. The curve $Y_{\tau}(4)$ has one rational component (corresponding to the trace of the critical 4-cycle of $\phi$), and one component $V$ of genus 2; moreover, the natural map $Y_1(4)\to Y_{\tau}(4)$ restricts to a map $U\to V$. The Jacobian of $V$ has rank 1, so the sets $V(\Q)$ and hence $U(\Q)$ can be computed. We find that all rational points on $U$ have $c\in\{0,\pm 1\}$, and \eqref{one_fourcycle} follows. 

The curve $\phi=A$ has genus 2 and its Jacobian has rank 1; we compute its rational points and find that every point has $c\in\{0,\pm 1\}$. Hence, $A(c)$ does not have rational preimages under $\phi_c$. The curve $\phi=Q$ is a hyperelliptic curve of genus 3, and its Jacobian has rank at most 1. The \texttt{Chabauty} functionality of Balakrishnan and Tuitman \cite{balakrishnan-tuitman} successfully computes all rational points on this curve, and every point has $c\in\{0,\pm 1\}$. Hence, $Q(c)$ does not have rational preimages under $\phi_c$. The curve $\phi=1$ is computed to be the curve in $\A^2=\Spec\Q[t,x]$ given by $x=0$; hence, $0$ is the only rational preimage of $1$ under $\phi_c$. This proves \eqref{AQ1_preimage}.

Now assume \eqref{diophantine}. If $\phi_c$ has a rational point of type $(2,2)$, then $c$ is the projection of a rational point on $Y_1(2,2)$; this is a contradiction, as the five known points have $c$-coordinate in the set $\{0,\pm 1\}$. Similarly, suppose that $\phi_c$ has both a $2$-periodic point and a point of type $(2,4)$ defined over $\Q$. By \eqref{one_fourcycle}, a rational point of type $(2,4)$ must map into the copy of I1 after two iterations of $\phi_c$; hence, by \eqref{AQ1_preimage}, such a point must be a preimage of $\infty$. It follows that $c$ is the projection of a rational point on the fiber product $F$, and therefore $c\in\{0,\pm 1,5/2\}$. This proves \eqref{2234}.
\end{proof}

\begin{figure}[ht]
\centering
\includegraphics[scale=0.45]{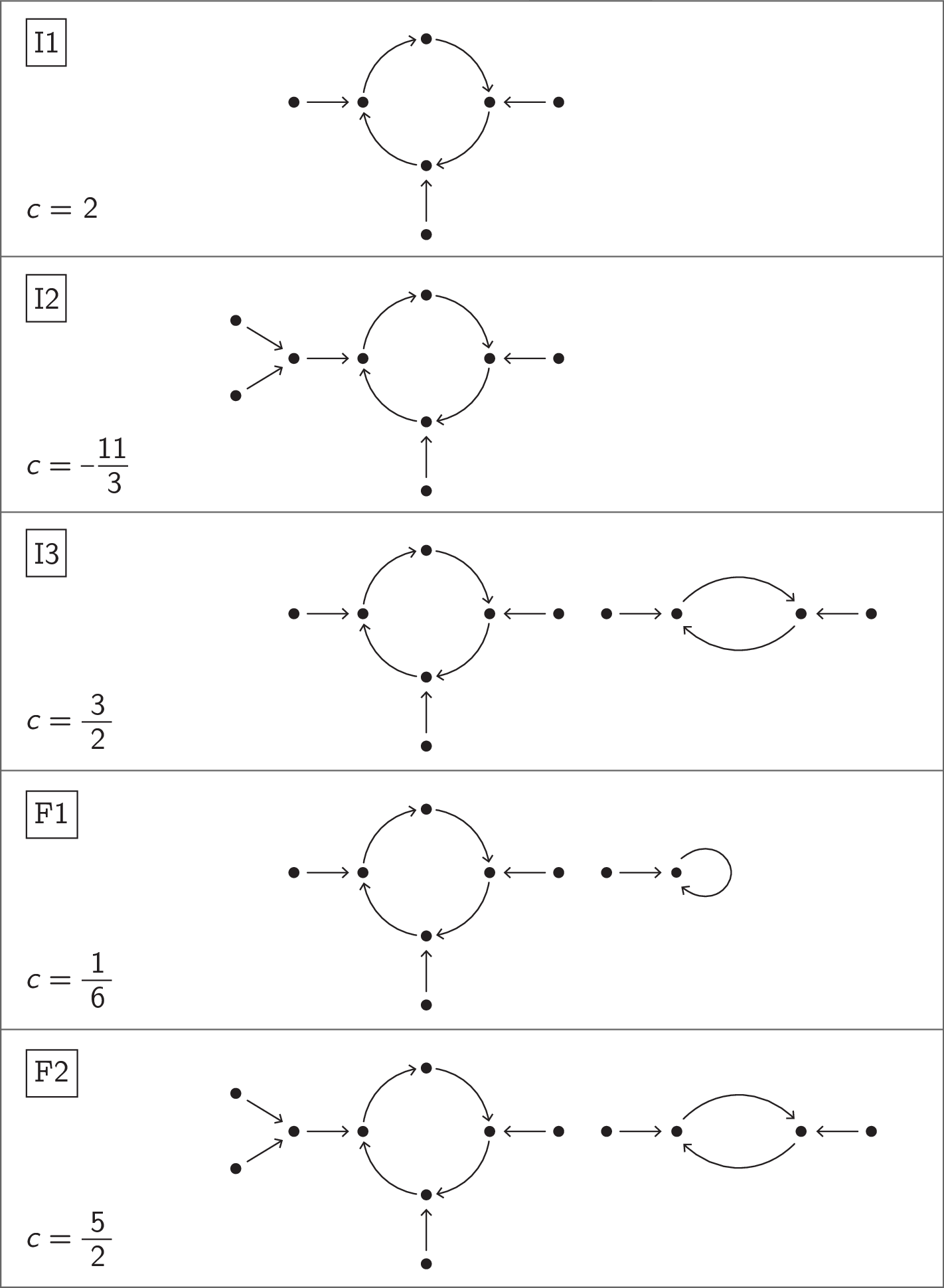}
\caption{Possibilities for the portrait $G(\phi_c,\Q)$, where $\phi$ is given by \eqref{period4_functionfield}, each with a representative parameter $c$.}
\label{period4portraits_trivial}
\end{figure}

The following proposition completes the proof of Theorem \ref{main_thm}.

\begin{prop}\label{period4_finiteness}
Let $\phi\in\Rat_2(\Q)$, and suppose that $\phi$ has a $4$-periodic critical point in $\PP^1(\Qbar)$. Assuming that $\phi$ has no rational periodic point of period greater than $4$, and assuming \eqref{diophantine}, the portrait $G(\phi,\Q)$ must be one of the five shown in Figure \ref{period4portraits_trivial}.
\end{prop}

\begin{proof}
By Proposition \ref{parametrization4}, we have $\langle\phi\rangle=\langle\phi_c\rangle$ for some $c\in\Q\setminus\{0,\pm 1\}$. Letting $G=G(\phi_c,\Q)$, we must show that $G$ is one of the portraits in Figure \ref{period4portraits_trivial}. The desired result holds true if $c\in\{1/6,5/2\}$, as $G$ will then be either F1 or F2. Thus, we assume henceforth that $c\notin\{1/6,5/2\}$. 

As noted earlier, $G$ contains the portrait I1. Furthermore, the hypotheses of the proposition imply that every cycle in $G$ has length at most 4, and Proposition \ref{preimage_curves} implies that $G$ contains neither a 1-cycle nor a 3-cycle, and contains exactly one 4-cycle;  moreover, it contains either no 2-cycle or exactly one 2-cycle. It follows that $G$ has at most two connected components, namely the component containing I1, and possibly one component containing a 2-cycle; the latter component would necessarily consist of the 2-cycle together with the preimages of points in the cycle, by Proposition \ref{preimage_curves}\eqref{2234}.

Let $\Gamma$ be the component of $G$ containing I1. We will show that $\Gamma$ must be either I1 or I2, thus completing the proof of the proposition. Assuming that $\Gamma$ properly contains I1, Proposition \ref{preimage_curves}\eqref{AQ1_preimage} implies that every point in $\Gamma\setminus$ I1 must eventually map to $\infty$. Moreover, \eqref{diophantine} implies that no point in $\Gamma$ can map to $\infty$ after two iterations, as every rational point on the curve $Y(2,\infty)$ has $c\in\{0,\pm 1\}$. It follows that $\Gamma\setminus$ I1 agrees with the preimage of $\infty$ under $\phi_c$, and thus $\Gamma=$ I2, as claimed.
\end{proof}

\begin{rem}\label{period4_rem}
A search for rational points of small height on the curves $Y_{\tau}(n)$ for $n\in\{5,6,7\}$ produced only points with $c\in\{0,\pm 1\}$. Together with the computation carried out in Proposition \ref{portrait_search}, this suggests that maps $\phi$ as in Proposition \ref{period4_finiteness} cannot have rational points of period greater than $4$.
\end{rem}

\subsection{Infinitely occuring portraits} We end this article by proving Theorem \ref{infinitely_represented}, which addresses the question of which of the five portraits in Figure \ref{period4portraits_trivial} arise in the form $G(\phi,\Q)$ for infinitely many distinct arithmetical dynamical systems $\langle\phi\rangle\in\langle C_4\rangle$. Equivalently, by Proposition \ref{parametrization4}, we wish to determine which of these portraits occur as $G(\phi_c,\Q)$ for infinitely many distinct $c\in\Q$.

\begin{prop}\label{infinitude}
With notation as in Figure \ref{period4portraits_trivial}, the following hold.
\begin{enumerate}[(a)]
    \item\label{F1F2_finite} There exist only finitely many $c\in\Q$ such that $G(\phi_c,\Q)=$ {\rm F1} or {\rm F2}.
    \item\label{I123_infinite} Assume \eqref{diophantine}, and assume that no portrait $G(\phi_c,\Q)$ contains a cycle of length greater than $4$. Then, for each portrait $I\in\{\mathrm{I1, I2, I3}\}$, there exist infinitely many $c\in\Q$ such that $G(\phi_c,\Q)=I$.
\end{enumerate}
\end{prop}
\begin{proof}
    It follows from Proposition \ref{preimage_curves}\eqref{no_fixed_period3} that F1 occurs as $G(\phi_c,\Q)$ for the unique value $c=1/6$. Moreover, as noted in the proof of Proposition \ref{preimage_curves}\eqref{2234}, if $\phi_c$ has both a $2$-periodic point and a point of type $(2,4)$ defined over $\Q$, then $c$ is the projection of a rational point on the fiber product $F$. This curve has genus 5, so its set of rational points is finite; hence $G(\phi_c,\Q)$ occurs as F2 for only finitely many $c$, proving \eqref{F1F2_finite}.
    
    The curves $Y(1,\infty)$ and $Y_1(2)$ have infinitely many rational points, as both are birational to elliptic curves of Mordell--Weil rank 1. Let $A$ and $B$, respectively, denote the projections (via the $c$-map) of the sets of rational points on these two curves. Then $A$ and $B$ are infinite thin subsets of $\A^1(\Q)$ in the sense of Serre \cite{serre_topics}*{Chapter 3}. Moreover, since the curve $F$ has only finitely many rational points, the set $A\cap B$ is finite. It follows that the sets $\Q\setminus (A\cup B)$, $A\setminus B$, and $B\setminus A$ are all infinite.
    
    In what follows we exclude $c\in\{0,\pm 1,1/6\}$. If $c\in\Q\setminus (A\cup B)$, Proposition \ref{period4_finiteness} implies that $G(\phi_c,\Q)=$ I1, as $c$ is not the projection of a rational point on either $Y(1,\infty)$ or $Y_1(2)$. Similarly, if $c\in A\setminus B$, then $G(\phi_c,\Q)=$ I2, and if $c\in B\setminus A$, then $G(\phi_c,\Q)=$ I3. This completes the proof.
\end{proof}

\appendix
\section{Proof of Lemma \ref{period2_realizations}}\label{realization_appendix}
The following lemma was used in Section \ref{period2_3_section} to find instances of the portraits in Figure \ref{period2portraits}. Though the full statement of the lemma was not required, we include this result as it may be of interest to the reader wishing to characterize the parameters $c$ that give rise to each portrait.

\begin{lem}
With notation as in Proposition \ref{parametrization2}\eqref{psi_v_critical}, the following hold.
\begin{enumerate}[(a)]
    \item\label{a_rational} If $c=(q^3+1)^2/q^3$ for some $q\in\Q\setminus\{0,\pm 1\}$, then $\alpha$ is rational and $c\alpha-1$ has a rational cube root. The converse holds as well.
    \item\label{a_irrational} If $\alpha$ is irrational and $c\alpha-1$ has a cube root in $\Q(\alpha)$, then there exists $p\in\Q\setminus\{0,\pm 1,1/2,1/3\}$ such that $c=(4p^3-3p+1)/p^3$. The converse holds if and only if $(p+1)(1-3p)$ is not a rational square.
\end{enumerate}
\end{lem}

\begin{proof}
    If $c=(q^3+1)^2/q^3$ for some $q\in\Q\setminus\{0,\pm 1\}$, then $c\notin\{0,4\}$, as required in the definition of the maps $\psi_c$ from Proposition \ref{parametrization2}. Moreover, the definition of $\alpha$ yields $\alpha=q^3/(q^3 + 1)\in\Q$ and $c\alpha-1=q^3$. Conversely, suppose that $\alpha$ is rational and $c\alpha-1$ has a rational cube root. Then we have $c(c-4)=x^2$ for some $x\in\Q$. Parametrizing the conic defined by this equation, we conclude that there exists $s\in\Q\setminus\{0\}$ such that $c=(s + 1)^2/s$ and $x=(s^2 - 1)/s$, and thus $\alpha=s/(s + 1)$, $c\alpha-1=s$. It follows that $s=q^3$
 for some $q\in\Q\setminus\{0\}$, and therefore $c=(s + 1)^2/s=(q^3+1)^2/q^3$. We must have $q\ne\pm 1$ so that $c\notin\{0,4\}$. This proves \eqref{a_rational}.
 
Now suppose that $\alpha\notin\Q$ and $c\alpha-1$ has a cube root $p+q\alpha\in\Q(\alpha)$, where $p,q\in\Q$ and $q\ne 0$. The equation $(p+q\alpha)^3=c\alpha-1$ implies that
    \[0=c(p^3+1)-3pq^2-q^3=c^2-(3p^2q+3pq^2+q^3)c+q^3.\]
We cannot have $p=-1$, as the above equations would then imply that $c^2-9c+27=0$, a contradiction. Similarly, we cannot have $p=0$, as the equations would then imply $q=0$. In addition, we must have $p\notin\{1/2,1/3\}$, so that $c\notin\{0,4\}$. The first equation yields $c=(3pq^2+q^3)/(p^3+1)$; substituting this in the second equation, we obtain
\[(2p^2+pq+p-1)(p^2q^2+4p^3q+4p^4-p^2q-2p^3+3p^2+pq+p+1)=0.\]
We claim that the right-hand factor in this equation is never zero for $p\neq-1$. Assuming this claim for the moment, we have $2p^2+pq+p-1=0$, and therefore $q=(1-p-2p^2)/p$, from which it follows that
\[c=\frac{3pq^2+q^3}{p^3+1}=\frac{4p^3-3p+1}{p^3},\]
proving the first statement of \eqref{a_irrational}. To prove the claim, note that the curve defined by the right-hand factor is irreducible over $\Q$ and reducible over $\Qbar$, so every rational point on the curve must be singular. However, the only rational singular point has $p=-1$, so the claim follows.

Conversely, suppose that $c=(4p^3-3p+1)/p^3$ with $p\in\Q\setminus\{-1,0,1/2,1/3\}$ such that $(p+1)(1-3p)$ is not a rational square. The relation
    \[p^6c(c-4)=(p+1)(1-3p)(2p-1)^2\] then implies that $c(c-4)$ is not a 
    rational square, so $\alpha$ is irrational. Moreover, letting $q=(1-p-2p^2)/p$, we obtain $(p+q\alpha)^3=c\alpha-1$.
\end{proof}

\end{document}